\documentclass[12pt,leqno]{amsart}  
\usepackage{amsmath,amstext,amsthm,amssymb,amsxtra}
\usepackage[top=1.5in, bottom=1.5in, left=1.25in, right=1.25in]	{geometry}
\usepackage{txfonts} 
\usepackage[T1]{fontenc}
\usepackage{lmodern}

 \usepackage{euler}   
\usepackage{pdfsync}

\usepackage{float} 

\usepackage{tikz}

\allowdisplaybreaks

\usepackage{mathtools}
\mathtoolsset{showonlyrefs,showmanualtags}

\usepackage{hyperref} 
\hypersetup{
    colorlinks=true,       
    linkcolor=blue,          
    citecolor=magenta,        
    filecolor=magenta,      
    urlcolor=cyan           
}

\usepackage[msc-links]{amsrefs}  

\theoremstyle{plain} 
\newtheorem{lemma}[equation]{Lemma} 
 
\newtheorem{theorem}[equation]{Theorem} 
\newtheorem{corollary}[equation]{Corollary}

\theoremstyle{definition}

\theoremstyle{remark}

\newcommand{\unit}{1\!\!1}

\newcommand{\norm}[1]{\ensuremath{\left\|#1\right\|}}
\newcommand{\abs}[1]{\ensuremath{\left\vert#1\right\vert}}
\newcommand{\ip}[2]{\ensuremath{\left\langle#1,#2\right\rangle}}

\newcommand{\D}{\mathbb{D}}
\newcommand{\B}{\mathbb{B}}

\newcommand{\C}{\mathbb{C}}

\newcommand{\avg}[1]{\langle #1 \rangle}

\newcommand{\vol}[1]{\textnormal{vol}_{#1}}
\newcommand{\Bn}{\mathbb{B}_n}

\renewcommand{\v}[1]{v_{#1}}
\newcommand{\id}{\textnormal{id}}
\newcommand{\sn}{\mathbb{S}_{N\theta}}

\renewcommand{\tanh}{\textnormal{tanh}}

\begin{document}

\title[Berezin and Bergman Estimates] 
  {Weighted Berezin and Bergman Estimates on the Unit Ball in $\C^{n}$}

\author[R. Rahm]{Rob Rahm}
\address{Robert Rahm, Washington University in St. Louis Department of Mathematics
  \\One Brookings Drive 
  \\St. Louis, MO 63130} 
  \email{robertrahm@gmail.com}

\author[E. Tchoundja]{Edgar Tchoundja}
\address{Edgar Tchoundja, Washington University in St. Louis Department of Mathematics
  \\One Brookings Drive 
  \\St. Louis, MO 63130} 
  \email{etchoundja@math.wustl.edu}

\author[B. D. Wick]{Brett D. Wick}
\address{Brett D. Wick, Washington University in St. Louis Department of Mathematics
  \\One Brookings Drive 
  \\St. Louis, MO 63130} 
  \email{wick@math.wustl.edu}
\thanks{B. D. Wick's research supported in part by National Science Foundation DMS grants \#0955432 and \#1560955.}

\maketitle

\begin{abstract}
Using modern techniques of dyadic harmonic analysis, we are able to
prove sharp estimates for the Bergman projection and Berezin 
transform and more general operators in weighted Bergman spaces
on the unit ball. 
The estimates are in terms of the Bekolle-Bonami constant of 
the weight. 
\end{abstract}

\section{Introduction and main results} 
Recall that the Bergman space $A_t^{p}(\Bn)=:A_t^p$ is defined to be
the space of holomorphic functions on $\Bn$ with finite 
$L_t^p(\Bn):=L_t^p$ norm. That is $f\in A_t^p$ if it is holomorphic 
and the following norm is finite:
\begin{align*}
\norm{f}_{A_t^p}^{p}
:=c_t\int_{\Bn}\abs{f(z)}^{p}(1-\abs{z}^2)^{t}dV(z).
\end{align*}
Above, $dV(z)$ is the standard Lebesgue measure on $\Bn$ and 
for $t> -1$, the constant $c_t$ is chosen so that 
$\int_{\Bn}c_t(1-\abs{z}^2)^{t}dV(z)=1$. When $t\leq -1$, 
we set $c_t=1$. We will let $d\v{t}=c_t(1-\abs{z}^2)^{t}dV(z)$. 

The purpose of this paper is to prove one--weight inequalities for the 
operators given by:
\begin{align*}
S_{a,b}f(z)
:=\left(1-\abs{z}^2\right)^{a}
	\int_{\Bn}\frac{f(w)}{\left(1-z\overline{w}\right)^{n+1+a+b}}d\v{b}(w)
\end{align*}
and 
\begin{align*}
S_{a,b}^{+}f(z)
:=\left(1-\abs{z}^2\right)^{a}
	\int_{\Bn}\frac{f(w)}{\abs{1-z\overline{w}}^{n+1+a+b}}d\v{b}(w),
\end{align*}
where $-a<b+1$ and $z\overline{w}=\sum_{i=1}^{n}z^i\overline{w}^i$.
That is, we want to know for which weights, i.e. positive locally integrable 
functions $u$, we have the following norm inequality:
\begin{align*}
\norm{S_{a,b}:L_{b}^{p}(u)\to L_{b}^{p}(u)}<\infty 
\end{align*}
where $L^{p}_{b}(u)$ denotes the set of functions that are $p$th power integrable with respect to $u(z)dv_b(z)$.

The operators $S_{a,b}$ and $S_{a,b}^{+}$ are important in the study of 
function--theoretic operator theory on the Bergman spaces (see for example, 
\cite{Zhu2005}) and so are interesting in their own right. However, 
our main motivation comes from the operators $S_{0,b}$ and $S_{n+1+b,b}^{+}$
which are the Bergman projection and Berezin transform respectively. 



Before we state our main result, we need to give some definitions. 
Recall that for $z\neq 0$, the Carleson tent over $z\in\Bn$ is defined 
to be the set:
\begin{align*}
T_{z}:=
\left\{w\in\Bn: \abs{1-\overline{w}\frac{z}{\abs{z}}}<1-\abs{z}\right\}
\end{align*}
and the Carleson tent over $0$ is $\Bn$. For $b>-1$, we define the $D_{p,a,b}$ 
characteristic of two weights $u,\sigma$ by:
\begin{align}
[u,\sigma]_{D_{p,a,b}}
&:=\sup_{z\in\Bn}
 \left(\frac{\int_{T_z}\sigma d\v{b}}{\int_{Tz}d\v{b}}\right)^{p-1}
 \frac{\int_{T_z}u{d\v{pa+b}}}{\int_{T_z}d\v{pa+b}}
\\&\simeq\sup_{z\in\Bn}
 \left(\frac{\int_{T_z}\sigma d\v{b}}{\int_{Tz}d\v{b}}\right)^{p-1}
 \frac{\int_{T_z}\widetilde{u}{d\v{b}}}{\int_{T_z}d\v{b}}
 \vol{b}(T_z)^{\frac{-pa}{n+1+b}},
\end{align}
where $\widetilde{u}(z):=u(z)(1-\abs{z}^2)^{pa}$. Using the notation 
we will use in this paper (defined below), we can write this 
more compactly as:
\begin{align}\label{E:dpdef}
[u,\sigma]_{D_{p,a,b}}
=\sup_{z\in\Bn}
 {
 \left(\avg{\sigma}_{T_z}^{d\v{b}}\right)^{p-1}
 \avg{u}_{T_z}^{d\v{pa+b}}}
\simeq\sup_{z\in\Bn}
 {
 \left(\avg{\sigma}_{T_z}^{d\v{b}}\right)^{p-1}
 \avg{\widetilde{u}}_{T_z}^{d\v{b}}}
 {\vol{b}(T_z)^{\frac{-pa}{n+1+b}}}.
\end{align}

Our main theorem is:
\begin{theorem}\label{T:1wtsat}
Let $1<p<\infty$ and let $u$ be a weight and let 
$\sigma=u^{\frac{-p'}{p}}$ be the dual weight. If $b>-1$ there holds:
\begin{align*}
[u,\sigma]_{D_{p,a,b}}^{\frac{1}{2p}}
&\lesssim \norm{S_{a,b}:L_{b}^p(u)\to L_{b}^p(u)}
\\&\leq \norm{S_{a,b}^{+}:L_{b}^p(u)\to L_{b}^p(u)}
\lesssim [u,\sigma]_{D_{p,a,b}}^{\max\left\{1,\frac{1}{p-1}\right\}}.
\end{align*}
If $b\leq-1$, let 
$\psi(z)=u(z)^{\frac{-p'}{p}}(1-\abs{z}^2)^{\frac{-1}{p}(p'b+pa)}$ and
$\nu(z)=\psi(z)^{\frac{-p}{p'}}$. There holds:
\begin{align}
[\psi,\nu]_{D_{p',b,a}}^{\frac{1}{2p'}}
&\lesssim \norm{S_{a,b}:L_{b}^{p}(u)\to L_{b}^{p}(u)}
\\&\leq \norm{S_{a,b}^{+}:L_{b}^{p}(u)\to L_{b}^{p}(u)}
\lesssim [\psi,\nu]_{D_{p',b,a}}^{\max\left\{1,\frac{1}{p'-1}\right\}}.
\end{align}
\end{theorem}

The classical $B_p$ characteristic of a weight is 
$[u]_{B_{p,b}}=[u,\sigma]_{D_{p,0,b}}$ where $\sigma=u^{\frac{-p'}{p}}$. 
Therefore, as a corollary of Theorem \ref{T:1wtsat} we have the 
following theorem, which is new for $n\geq 2$:
\begin{theorem}\label{T:1wtproj}
Let $1<p<\infty$ let $u$ be a weight and let $P_b=S_{0,b}$ be the 
Bergman projection. There holds:
\begin{align*}
[u]_{B_{p,b}}^{\frac{1}{2p}}
\lesssim \norm{P_{b}:L_{b}^p(u)\to L_{b}^p(u)}
\lesssim [u]_{B_{p,b}}^{\max\left\{1,\frac{1}{p-1}\right\}}.
\end{align*}
\end{theorem}

Restricting attention to $\mathcal{B}_{b}:=S_{n+1+b,b}$, we define 
$[u]_{C_{p,b}}:=[u,\sigma]_{D_{p,n+1+b,b}}$. As a corollary of 
Theorem \ref{T:1wtsat} we have:
\begin{theorem}\label{T:1wtber}
Let $1<p<\infty$ let $u$ be a weight and let $\mathcal{B}_{b}:
=S_{n+1+b,b}$ be the Berezin transform. There holds:
\begin{align*}
[u]_{C_{p,b}}^{\frac{1}{2p}}
\lesssim \norm{\mathcal{B}_{b}:L_{b}^p(u)\to L_{b}^p(u)}
\lesssim [u]_{C_{p,b}}^{\max\left\{1,\frac{1}{p-1}\right\}}.
\end{align*}
\end{theorem}

The following corollary of Theorem \ref{T:1wtsat} is well--known. 
See for example, \cite{Zhu2005,HedKorZhu2000}.
\begin{corollary}
For $1<p<\infty$ the operator $S_{a,b}$ is bounded from $L_{t}^{p}$ to itself 
if and only if $-pa< t+1 < p(b+1)$.
\end{corollary}
The proof is to take $u(z)=(1-\abs{z}^2)^{t-b}$ and to note that the integrals 
in the definition of the $D_{p,a,b}$ condition are finite if and only if 
$-pa< t+1 < p(b+1)$. The details are left to the reader.

%
%

It is also well--known by now that $P_{b}$ is bounded from 
$L_{b}^{p}(u)$ to itself if any only if $u$ is a $B_{p,b}$ 
weight. This was proven for the disc in \cite{BekBom1978} 
and for the ball in \cite{Bek1981}. The sharp dependence of 
the operator norm on the $B_{p,b}$ characteristic was given 
by S. Pott- M.C. Reguera in \cite{PottReg2013} for the Bergman space 
on the disc; namely the case when $n=1$. 


Our technique is that of dyadic operators. We show that the operators 
of interest can be dominated by positive dyadic operators and we use the 
techniques of modern dyadic harmonic analysis to deduce the desired 
estimates. This is the approach that S. Pott-- M. C. Reguera took in \cite{PottReg2013}, 
though we use a recent but similar approach of Lacey \cite{Lac2015} that 
avoids an extrapolation argument.


The outline of the paper is as follows. In Section \ref{bckgrnd} we briefly 
give requisite background information and we recall  
a dyadic structure for $\Bn$ given in, for example, \cite{ArcRocSaw2006}.
In Section \ref{dom}, we show that $S_{a,b}^{+}$ is equivalent to a finite sum 
of dyadic operators and in Section \ref{proofs}, we prove Theorem \ref{T:1wtsat}.
Section \ref{sharp} contains an example showing that the upper bound 
in Theorem \ref{T:1wtsat} is sharp. Finally, Section \ref{concl} contains 
concluding remarks.

\section{Background Information and Notation}\label{bckgrnd}
The following notation will be used throughout the paper. 
For a weight $u$ and a subset $E\subset \Bn$, we set $u_t(E)=\int_{E}
u(z)d\v{t}(z)$ and $\vol{t}(E)=\int_{E}d\v{t}$. For a measure, $\mu$, 
and a subset $E\subset \Bn$ we define 
$\avg{f}_{E}^{d\mu}:=\frac{1}{\mu(E)}\int_{E}f(z)d\mu(z)$.

We begin by recalling some geometric facts on the ball $\Bn$. 
Let $\varphi_z$ be the involutive automorphism of $\Bn$ that 
interchanges $z$ and $0$. That is, $\varphi_z$ is a holomorphic 
function from $\Bn$ to itself that satisfies $\varphi_z\circ\varphi_z=\id$, $\varphi_z(0)=z$, and $\varphi_z(z)=0$. Using the maps $\varphi_z$ we
can define the so--called Bergman metric, $\beta$ on $\Bn$, by:
\begin{align*}
\beta(z,w)=\frac{1}{2}\log\frac{1+\abs{\varphi_z(w)}}
  {1-\abs{\varphi_z(w)}}.
\end{align*}
Let $B_{\beta}(z,r)$ be the ball in the Bergman metric of radius 
$r$ centered at $z$. It is well--known (see for example, \cite{Zhu2005})
that for $w\in B_{\beta}(z,r)$ there holds:
\begin{align}\label{E:volstuff}
\vol{t}(B_{\beta}(z,w))
\simeq \abs{1-\overline{z}w}^{n+1+t}
\simeq \left(1-\abs{z}^2\right)^{n+1+t}
\simeq \left(1-\abs{w}^2\right)^{n+1+t}.
\end{align}
It is worthwhile to note that we will make heavy use of this and 
similar estimates. 

We next introduce a dyadic structure on the ball. The 
construction we use is the one given in, for example, \cite{ArcRocSaw2006}. 
We start by fixing two parameters, $\theta,\lambda>0$. These 
parameters will roughly correspond to the ``sizes'' of Carleson 
boxes.

For $N\in\mathbb{N}$, 
let $\mathbb{S}_{N\theta}$ be the sphere of radius $N\theta$ in the 
Bergman metric. We can find a sequence of points, $E_N=\{w_j\}_{j=1}^{J_N}$
and a corresponding sequence of Borel subsets, $\{Q_{j}^{N}\}_{j=1}^{J_N}$ of 
$\sn$ that satisfy:
\begin{align}\label{setsonsphere}
 &(i)\hspace{.1in} \sn=\cup_{j=1}^{J_N}Q_{j}^{N},\\
 &(ii) Q_{j}^{N}\cap Q_{i}^{N}=\emptyset
 \textnormal{ when } i\neq j,\\
 &(iii) \sn\cap B_{\beta}(w_j,\lambda)\subset Q_{j}^{N}
  \subset \sn\cap B_{\beta}(w_j,C\lambda).
\end{align}

Let $P_{N\theta}z$ be the radial projection of $z$ onto the 
sphere $\sn$. Define subsets, $K_j^{N}$ of $\Bn$ by:
\begin{align*}
K_{1}^{0}&:=\{z\in\Bn:\beta(0,z)<\theta\} \\
K_{j}^{N}&:=\{z\in\Bn:N\theta\leq\beta(0,z)<(N+1)\theta 
  \textnormal{ and } P_{N\theta}z\in Q_{j}^{N}\}, N\geq 1, j\geq 1.
\end{align*}
Now, let $c_j^{N}\in K_{j}^{N}$ be defined by 
$P_{(N+\frac{1}{2})\theta}w_{j}^{N}$. The sets $K_j^N$ are 
referred to as $kubes$ and the points $c_j^N$ are the centers of 
the kubes.

Now we define a tree structure $\mathcal{T}:=\{c_j^N\}$ on the 
centers of the kubes. We say that $c_{i}^{N+1}$ is a child of 
$c_{j}^{N}$ if $P_{N\theta}c_{i}^{N+1}\in Q_{j}^{N}$. 

We will denote elements of the tree by the letters $\alpha$ and $\beta$
and $K_\alpha$ will be the kube with center $\alpha$. We 
will also abuse notation and use, for example, $\alpha$ to denote
both an element of a tree $\mathcal{T}$ and the center of the 
corresponding kube or, in fact, any convenient element of 
the kube. 
There is the usual partial 
order on the tree: if $\alpha,\beta\in\mathcal{T}$ we say that 
$\beta\geq\alpha$ if $\beta$ is a descendant of $\alpha$. 
We will use $\widehat{K_\alpha}$ to be the dyadic tent under $K_\alpha$.
That is:
\begin{align*}
\widehat{K_\alpha}
:=\bigcup_{\beta\in\mathcal{T}:\beta\geq\alpha}K_\beta.
\end{align*}
We will
also use $d(\alpha)$ to denote the ``generation'' of $\alpha$, 
or the distance in the tree from $\alpha$ to the root. Thus, 
if $N\theta<\beta(0,\alpha)<(N+1)\theta$, then 
$d(\alpha)=N$.

We have the following lemma proven in \cite{ArcRocSaw2006}.
\begin{lemma}\label{L:tree}
Let $t>-1$ and 
let $\mathcal{T}$ be a tree constructed with positive parameters
$\lambda$ and $\theta$. Then the tree satisfies the following 
properties:
\begin{itemize}
 \item [(i)] $\Bn=\cup_{\alpha\in\mathcal{T}}K_\alpha$ and 
 the kubes $K_\alpha$ are pairwise disjoint. Furthermore, there
 are constants, $C_1$ and $C_2$ depending on $\lambda$ and $\theta$
 such that for all $\alpha\in\mathcal{T}$ there holds:
 \begin{align*}
  B_{\beta}(\alpha,C_1)\subset K_\alpha
  \subset B_{\beta}(\alpha, C_2), 
 \end{align*}
 \item[(ii)] $\vol{t}\left(\widehat{K_\alpha}\right)
 \simeq \vol{t}\left(K_{\alpha}\right)$,
 \item[(iii)] $\vol{t}(T_\alpha)\simeq\vol{t}(\widehat{K_\alpha})\simeq
 \vol{t}(K_\alpha)\simeq (1-\abs{\alpha}^2)^{n+1+t}$, 
 \item[(iv)] Every element of $\mathcal{T}$ has at most $e^{2n\theta}$
 children. 
\end{itemize}
\end{lemma}

Note that if $z\in S_{N\theta}$, then 
$N\theta=\frac{1}{2}\log\frac{1+\abs{z}}{1-\abs{z}}$ and therefore, 
$\abs{z}=\frac{e^{2N\theta-1}}{e^{2N\theta}+1}$.
For the rest of the paper, set $r_{N\theta}
:=\frac{e^{2N\theta}-1}{e^{2N\theta}+1}$. Note that we 
have $1-r_{N\theta}\simeq e^{-2N\theta}$. Therefore, if 
$\beta(0,\alpha)=(N+\frac{1}{2})\theta$ (that is, $d(\alpha)=N$), we have: 
\begin{align*}
(1-\abs{\alpha}^2)^{n+1+t}
= \left(1-\left(\frac{e^{2(N+\frac{1}{2})\theta}-1}
  {e^{2(N+\frac{1}{2})\theta}+1}\right)^2\right)^{n+1+t}
\simeq e^{-2(N+\frac{1}{2})\theta(n+1+t)}
\simeq e^{-2N\theta(n+1+t)}.
\end{align*}
Therefore, if $d(\alpha)=N$, there holds:
\begin{align*}
\vol{t}(T_\alpha)
\simeq \vol{t}(\widehat{K_\alpha})
\simeq \vol{t}(K_\alpha)
\simeq (1-\abs{\alpha}^2)^{n+1+t}
\simeq e^{-2N\theta(n+1+t)}.
\end{align*}


We now show that every Carleson tent is well--approximated by a dyadic tent. 
To do this, we start with a special case of a lemma from \cite{HytKai2012}.

Let $\rho$ be the pseudo--metric on $\partial{\Bn}$ given by $\rho(z,w)
=\abs{1-\overline{z}w}$. As usual, $D(z,r):=\{w\in\partial\Bn:
\rho(z,w)<r\}$. A system of dyadic cubes of calibre $\delta$
is a collection of Borel subsets $\mathcal{D}:=\{Q_{i}^{k}\}_{i,k\in\mathbb{Z}}$ 
and points $\{z_i^{k}\}_{i,k\in\mathbb{Z}}$ in $\partial{\Bn}$
that satisfy:
\begin{itemize}
 \item [(i)] There are constants $c_1,C_2$ such that for every 
 $k,i\in\mathbb{Z}$ there holds:
 \begin{align*}
  D(z_i^{k},c_1\delta^{k})\subset Q_{i}^{k} 
  \subset D(z_i^{k},C_{2}\delta^{k}). 
 \end{align*}
 \item [(ii)] For all $k\in\mathbb{Z}$ there holds $\partial{\Bn}
 =\cup_{i\in\mathbb{Z}}Q_{i}^{k}$ and the sets are disjoint.
 \item [(iii)] If $Q,R\in\mathcal{D}$ and $Q\cap R\neq \emptyset$, then 
 either $Q\subset R$ or $R\subset Q$. 
\end{itemize}

We have the following lemma which is a special case 
of \cite{HytKai2012}*{Theorem 4.1}. 
\begin{lemma}[Hyt\"onen and Kairema]\label{L:hytkai}
For every $\delta>0$ there is an $M\in\mathbb{N}$ such that 
there is a collection of dyadic systems of cubes 
$\{\mathcal{D}_l\}_{l=1}^{M}$ with the following property:
For every disc $D(z,r):=\{w\in\partial{\Bn}:\rho(z,w)<r\}$, 
there is a $1\leq t \leq M$ such that there is a dyadic cube 
$Q_{i}^{k}\in\mathcal{D}_{t}$ with $D(z,r)\subset Q_{i}^{k}$
and $\delta^{k}\simeq r$ where the implied constants are 
independent of $z$ and $r$. 
\end{lemma}

\begin{lemma}\label{L:covering}
There is a finite collection of Bergman trees $\{\mathcal{T}_l\}_{l=1}^{N}$
such that for all $z\in\Bn$, there is a tree $\mathcal{T}$ from the finite
collection and an $\alpha\in\mathcal{T}$ such that the dyadic tent
$\widehat{K_\alpha}:=\cup_{\beta\geq\alpha} K_\beta$ contains the tent $T_z$ and 
$\vol{t}(\widehat{K_\alpha})\simeq\vol{t}(T_z)$. 
\end{lemma}
\begin{proof}
Let $\mathcal{D}=\{Q_{i}^{k}\}_{i,k\in\mathbb{N}}$ be a dyadic system of calibre $\delta$. 
We will use this dyadic system to create a Bergman tree with parameters $\theta$ 
and $\lambda$ where $\delta=e^{-2\theta}$ and $\lambda$ will be chosen below. For 
each $k\in\mathbb{N}$, we project the sets $\{Q_{i}^{k}\}_{i\in\mathbb{Z}}$ 
radially onto the sphere $\mathbb{S}_{k\theta}$. Let $P_{k\theta}$ be the 
radial projection onto the sphere $\mathbb{S}_{k\theta}$. We will now show that 
these sets $\{P_{k\theta}Q_{i}^{k}\}_{i\in\mathbb{N}}$ satisfy the three properties 
in \eqref{setsonsphere} which means a Bergman tree can be constructed from 
them according to the construction in \cite{ArcRocSaw2006}. 

Clearly the sets $\{P_{k\theta}Q_{i}^{k}\}_{i\in\mathbb{N}}$  satisfy Properties 
(i) and (ii) in \eqref{setsonsphere}. For the third property, observe that it is 
enough to show that there are two positive constants $\lambda_1$ and $\lambda_2$ 
independent of $i,k$ that:
\begin{align}\label{E:inclusions}
\mathbb{S}_{k\theta}\cap B_\beta(r_{k\theta}z_{i}^{k},\lambda_1)
\subset P_{k\theta}Q_{i}^{k}
\subset \mathbb{S}_{k\theta}\cap B_\beta(r_{k\theta}z_{i}^{k},\lambda_2)
\end{align}
where $z_{i}^{k}$ is the centre of $Q_{i}^{k}$ and $P_{k\theta}Q_{i}^{k}$
is the projection of $Q_{i}^{k}$ onto $\mathbb{S}_{k\theta}$. Now, recall 
that $\tanh \beta(z,w) = \abs{\varphi_{z}(w)}$. Therefore, $\beta(z,w)\leq R$ 
if and only if $\abs{\varphi_{z}(w)} \leq \tanh R$ if and only if 
$1-\abs{\varphi_{z}(w)}^2 \geq 1-(\tanh{R})^2$. Now, $1-\abs{\varphi_{z}(w)}^2 = 
(1-\abs{z}^2)(1-\abs{w}^2) \abs{1-z\overline{w}}^{-2} = 
(1-r_{k\theta}^2)^2\abs{1-z\overline{w}}^{-2}$.
Thus, for $z,w\in\mathbb{S}_{k\theta}$,  $\beta(z,w)\leq R$ if and only if 
$\abs{1-z\overline{w}}\leq 
(1-r_{k\theta}^2)(1-(\tanh R)^2)^{-1/2}\simeq e^{-2k\theta}(1-(\tanh R)^2)^{-1/2}$.
Now, if $\xi\in Q_{i}^{k}$ then there holds:
\begin{align}\label{E:kuberhs}
\abs{1-P_{k\theta}\xi\overline{P_{k\theta}z_{i}^{k}}}
=\abs{1-r_{k\theta}\xi r_{k\theta}\overline{z_{i}^{k}}}
\leq \abs{1-r_{k\theta}^2}+r_{k\theta}^2\abs{1-\xi\overline{z_{i}^{k}}}
\simeq e^{-2k\theta}.
\end{align}
On the other hand, if $\xi\in \mathbb{S}_{k\theta}\cap B_{\beta}
(P_{k\theta}z_{i}^{k},R)$, 
then $\abs{1-r_{k\theta}\frac{\xi}{\abs{\xi}}r_{k\theta}\overline{z_{i}^{k}}}
\lesssim e^{-2k\theta}(1-(\tanh R)^2)^{-1/2}$ and so there holds:
\begin{align}\label{E:kubelhs}
\abs{1-\frac{\xi}{\abs{\xi}}\overline{z_{i}^{k}}}
\leq \abs{1-r_{k\theta}\frac{\xi}{\abs{\xi}}r_{k\theta}\overline{z_{i}^{k}}}
  + \abs{\frac{\xi}{\abs{\xi}}r_{k\theta}^2\overline{z_{i}^{k}}
  - \frac{\xi}{\abs{\xi}}z_{i}^{k}}
\simeq e^{-2k\theta}
=\delta^{k}.
\end{align}
Clearly, \eqref{E:kuberhs} and \eqref{E:kubelhs} together imply 
the existence of $\lambda_1$ and $\lambda_2$ such that 
\eqref{E:inclusions} is satisfied. 

Let $T_z$ be a Carleson tent and note that the ``base'' of 
$T_z$ is the disc $D(Pz,1-\abs{z})$. 
By Lemma \ref{L:hytkai}, there is a finite number of dyadic systems, 
$\{\mathcal{D}_l\}_{l=1}^{M}$, such that every disc $D$ is contained in a 
dyadic cube of comparable radius. Then this 
disc is contained in some element $Q$ of one of the dyadic systems and the 
dyadic tent over $Q$ contains the tent $T_z$. This completes the proof. 
\end{proof}

Of course, maximal functions with respect to this dyadic structure 
will play a role. Thus, for a weight, $u$, and a Bergman tree, $\mathcal{T}$, 
define the following maximal function:
\begin{align*}
M_{\mathcal{T},u}f(w)
  :=\sup_{\alpha\in\mathcal{T}}\frac{\unit_{\widehat{K_\alpha}}(w)}
  {u_t(\widehat{K_\alpha})}
  \int_{\widehat{K_\alpha}}\abs{f(z)}u(z)d\v{t}(z).
\end{align*}
The following lemma is well--known:
\begin{lemma}\label{L:maxbdd}
Let $-1 < t$ and $u$ be a weight, then $M_{\mathcal{T},u}$ is bounded on  
$L_t^p(u)$ for $1<p\leq\infty$ and is bounded from 
$L_t^1(u)\to L_{t}^{1,\infty}(u)$. 
\end{lemma}

Finally, we make an observation. The norm inequality 
$\norm{S_{a,b}f}_{L_{b}^{p}(u)}\lesssim 
\norm{f}_{L_{b}^{p}(u)}$ is the same as the norm inequality
$\norm{Q_{a,b}f}_{L_{b}^{p}(\widetilde{u})}\lesssim 
\norm{f}_{L_{b}^{p}(u)}$ where $\widetilde{u}(z):=u(z)(1-\abs{z}^2)^{pa}$ 
and 
\begin{align*}
Q_{a,b}f(z)
:=\int_{\Bn}\frac{1}{(1-z\overline{w})^{n+1+a+b}}
  f(w)d\v{b}(w).
\end{align*}
A similar remark is true for $S_{a,b}^{+}$ and the 
similarly defined operator $Q_{a,b}^{+}$. Thus for $b>-1$ the claim in 
Theorem \ref{T:1wtsat} is equivalent to:
\begin{align*}
[u,\sigma]_{D_{p,a,b}}^{\frac{1}{2p}}
\lesssim \norm{Q_{a,b}:L_{b}^p(u)\to L_{b}^p(\widetilde{u})}
\leq \norm{Q_{a,b}^{+}:L_{b}^p(u)\to L_{b}^p(\widetilde{u})}
\lesssim [u,\sigma]_{D_{p,a,b}}^{\max\left\{1,\frac{1}{p-1}\right\}}.
\end{align*}


\section{Equivalence to a Dyadic Operator}\label{dom}
In this section, we will show that when $b+1>0$, $Q_{a,b}^{+}$ is 
pointwise equivalent to a finite sum of simple operators of the form:
\begin{align}\label{E:sparse}
T_{\mathcal{T}}f
:=\sum_{\alpha\in\mathcal{T}}
  \vol{b}\left(\widehat{K_\alpha}\right)^{\frac{-a}{n+1+b}}
  \avg{f}_{\widehat{K_\alpha}}^{d\v{b}}
  \unit_{\widehat{K_\alpha}}
\end{align}
where $\mathcal{T}$ is a Bergman tree. 

We first show that $Q_{a,b}^{+}$ is dominated by a finite sum of operators 
of the desired form. Assume  
that for every $z,w\in\Bn$, there is a Carleson tent, $T$, containing
$z$ and $w$ such that $\vol{b} (T) \simeq \abs{1-\overline{z}w}^{n+1+b}$.
Then we can use Lemma \ref{L:covering} to deduce that
\begin{align*}
Q_{a,b}^{+}f(z)
\lesssim \sum_{l=1}^{M}\sum_{\alpha\in\mathcal{T}_l}
  \vol{b}\left(\widehat{K_\alpha}\right)^{\frac{-a}{n+1+b}}
  \avg{f}_{\widehat{K_\alpha}}^{d\v{b}}
  \unit_{\widehat{K_\alpha}}(z).
\end{align*}

Therefore, we only need to show that for every $z,w\in\Bn$, there is a 
Carleson tent $T$ containing $z,w$ such that $\vol{b} (T) \simeq 
\abs{1-\overline{z}w}^{n+1+b}$. We now turn to that.

Note that there is a $N\in {-1,0,1,2,\ldots}$ such that 
$\abs{1-\overline{z}w}\simeq e^{-2N\theta}$ (that is, 
$e^{-2(N+1)\theta}\leq\abs{1-\overline{z}w}<e^{-2N\theta}$). We will show
that there is a $k\in\mathbb{N}$ -- that does not depend on 
$z$ or $w$ -- such that $z,w\in T_{P_{(N-k)\theta}w}$.
Since $\vol{b} (T_{P_{(N-k)\theta}w})\simeq e^{-2(N-k)\theta(n+1+b)}
\simeq e^{-2N\theta(n+1+b)}\simeq \abs{1-\overline{z}w}^{n+1+b}$, 
this will prove the claim.

We first show that $w\in T_{P_{(N-k)\theta}w}$. Indeed, we will
show that $w\in T_{P_{N\theta}w}$. Since $w$ and $P_{N\theta}w$ 
are on the same ray, we need to show that $\abs{P_Nw}\leq\abs{w}$.
But this is not difficult:  
\begin{align*}
\abs{P_{N\theta}w}
=\frac{e^{2N\theta}-1}{e^{2N\theta}+1}
\leq 1-e^{-2N\theta}
\leq 1-\abs{1-\overline{z}w}
\leq 1- \abs{1-\abs{zw}}
=\abs{zw}
\leq \abs{w}.
\end{align*}

We next show that $z\in T_{P_{(N-k)\theta}w}$. To do this, it is enough to show 
that $\abs{1-\overline{z}Pw}\lesssim e^{-2N\theta}\simeq 1-\abs{P_{N\theta}w}$. (The $k$ will
essentially be the logarithm of the implied constant, but since the exact
value is not important, we do not attempt to calculate it.) First, note that 
$\abs{P_{N\theta}w}\leq\abs{w}$ and so there holds:
\begin{align*}
\abs{Pw-w} 
\leq\abs{Pw-P_{N\theta}w}
=\abs{\frac{w}{\abs{w}}-r_{N\theta}\frac{w}{\abs{w}}}
=1-r_{N\theta}
\simeq e^{-2N\theta}.
\end{align*}
And we then have: 
\begin{align*}
\abs{1-\overline{z}Pw}
\leq \abs{1-\overline{z}w} + \abs{\overline{z}Pw-\overline{z}w}
\lesssim e^{-2N\theta} + \abs{Pw-w}
\lesssim e^{-2N\theta}.
\end{align*}
Therefore, we have shown that $Q_{a,b}^{+}$ is dominated 
by a finite sum of operators of the form $T_\mathcal{T}$. 

We now show that $Q_{a,b}^{+}$ dominates every dyadic operator 
as we have defined above. That is, for every Bergman tree 
$\mathcal{T}$ we will show that for all $z\in\Bn$ there 
holds $\abs{T_{\mathcal{T}}f(z)}\lesssim Q_{a,b}^{+}\abs{f}(z)$.
The proof here is similar to the one in, for example, 
\cite{Cru2015} for the fractional integral operator. 

We first make some computations and fix some notation. 
First, we may assume that $f$ is non--negative. For 
$z\in\Bn$, let $\alpha=\alpha(z)$ be the unique element of $\mathcal{T}$
such that $z\in K_\alpha$. For $\beta\in\mathcal{T}$ with 
$\beta\geq\alpha$, let 
$s(\alpha,\beta)$ denote the unique element of $\Bn$ that 
satisfies $\beta \leq s(\alpha,\beta) \leq \alpha$ and 
$d(s(\alpha,\beta))=d(\beta)+1$. That is, $s(\alpha,\beta)$ is
the child of $\beta$ that is ``in--between'' $\beta$ and $\alpha$.
Let $E_{\alpha,\beta}:= \widehat{K_\beta}
\setminus\widehat{K_{s(\alpha,\beta)}}$ and observe that
for fixed $\alpha$, these sets are pairwise disjoint. 
Note that $\vol{b}(\widehat{K_\beta})\simeq e^{2\theta(n+1+b)}
\vol{b}(\widehat{K_{s(\alpha,\beta)}})$. Also, note that 
for $z,w\in\widehat{K_\beta}$, there holds 
$\abs{1-z\overline{w}}\lesssim 1-\abs{\beta}^2$. This can be 
seen by, for example, noting that $\abs{1-z\overline{w}}
\leq\abs{1-P\beta\overline{w}}+\abs{P\beta-z}$; since 
$z,w\in\widehat{K_\beta}$, both of these terms are dominated 
by $1-\abs{\beta}<1-\abs{\beta}^2$. Thus, for fixed 
$\alpha\in\mathcal{T}$ and $z\in K_\alpha$, there holds:
\begin{align}\label{E:compforint}
\sum_{\beta\in\mathcal{T}:\beta\leq\alpha}
  \int_{E_{\alpha,\beta}}
  \frac{f(w)d\v{b}(w)}
  {\vol{b}(\widehat{K_\beta})^{1+\frac{a}{n+1+b}}}
\lesssim \sum_{\beta\in\mathcal{T}:\beta\leq\alpha}
  \int_{E_{\alpha,\beta}}\frac{f(w)d\v{b}(w)}{\abs{1-z\overline{w}}^{n+1+a+b}}
\leq Q_{a,b}^{+}f(z).
\end{align}
Also, there holds (again for fixed $\alpha$ and $z\in K_\alpha$):
\begin{align}\label{E:compforsum}
\sum_{\beta\in\mathcal{T}:\beta\leq\alpha}
  \left(\vol{b}(\widehat{K_\beta})\right)^{-\frac{a}{n+1+b}-1}
  \int_{\widehat{K_{s(\alpha,\beta)}}}f(w)d\v{b}(w)
\end{align}
is controlled by
\begin{align}
e^{-2\theta(n+1+a+b)}\sum_{\beta\in\mathcal{T}:\beta\leq\alpha}
  \left(\vol{b}(\widehat{K_{s(\alpha,\beta)}})\right)^{-\frac{a}{n+1+b}-1}
  \int_{\widehat{K_{s(\alpha,\beta)}}}f(w)d\v{b}(w).
\end{align}
But this is just:
\begin{align*}
e^{-2\theta(n+1+a+b)}\sum_{\alpha\in\mathcal{T}}
  \vol{b}\left(\widehat{K_\alpha}\right)^{\frac{-a}{n+1+b}}
  \avg{f}_{\widehat{K_\alpha}}^{d\v{b}}
  \unit_{\widehat{K_\alpha}}(z)
=e^{-2\theta(n+1+a+b)}T_{\mathcal{T}}f(z).
\end{align*}
Therefore, for fixed $z\in\Bn$ and $\alpha=\alpha(z)$
\begin{align*}
T_{\mathcal{T}}f(z)
=\sum_{\gamma\in\mathcal{T}}
  \vol{b}\left(\widehat{K_\gamma}\right)^{\frac{-a}{n+1+b}}
  \avg{f}_{\widehat{K_\gamma}}^{d\v{b}}
  \unit_{\widehat{K_\gamma}}(z)
\end{align*}
is equal to
\begin{align*}
\sum_{\beta\in\mathcal{T}:\beta\leq\alpha}
  \int_{E_{\alpha,\beta}}
  \frac{f(w)d\v{b}(w)}
  {\vol{b}(\widehat{K_\beta})^{1+\frac{a}{n+1+b}}}
+ \sum_{\beta\in\mathcal{T}:\beta\leq\alpha}
  \left(\vol{b}(\widehat{K_\beta})\right)^{-\frac{a}{n+1+b}-1}
  \int_{\widehat{K_{s(\alpha,\beta)}}}f(w)d\v{b}(w).
\end{align*}
Therefore, by the above, we have:
\begin{align*}
T_{\mathcal{T}}f(z)
\leq CQ_{a,b}^{+}f(z) 
  +e^{-2\theta(n+1+a+b)}T_{\mathcal{T}}f(z).
\end{align*}
Since $n+1+a+b>0$, rearranging the above completes the proof. 

Thus, we have proven the following lemma:
\begin{lemma}\label{L:pwequiv}
There is a finite collection of Bergman trees, 
$\{\mathcal{T}_l\}_{l=1}^{M}$ such that for $b>-1$ there holds:
\begin{align*}
Q_{a,b}^{+}f(z)
\simeq\sum_{l=1}^{M}\sum_{\alpha\in\mathcal{T}_l}
  \vol{b}\left(\widehat{K_\alpha}\right)^{\frac{-a}{n+1+b}}
  \avg{f}_{\widehat{K_\alpha}}^{d\v{b}}
  \unit_{\widehat{K_\alpha}}(z).
\end{align*}
\end{lemma}

\section{Proof of Theorem \ref{T:1wtsat}}\label{proofs}
This section is devoted to the proof of Theorem \ref{T:1wtsat}.
The proof requires that either $b+1>0$ or $a+1>0$. Since 
$-a<b+1$ then this holds. The proofs given in this section 
are for the case $b+1>0$ and the case $b+1\leq0$ and $a>0$ 
is obtained by the following argument using duality. 

For a weight, $\omega$, the dual space of $L_{0}^{p}(\omega)$ under the 
unweighted inner product on $L_{0}^{2}$, is $L_{0}^{p'}(\omega^{\frac{-p'}{p}})$. 
It is also easy to see that the $L_{0}^{2}$ adjoint of $S_{a,b}$ is 
$S_{b,a}$. Let $\rho(z)=u(z)(1-\abs{z}^2)^{b}$ and let 
$\psi(z)=u(z)^{\frac{-p'}{p}}(1-\abs{z}^2)^{\frac{-1}{p}(p'b+pa)}$.
Note that $\psi(z)(1-\abs{z}^2)^{a}
=u(z)^{\frac{-p'}{p}}(1-\abs{z}^{2})^{\frac{-p'b}{p}}$ is the dual 
weight of $\rho(z)$. There holds:
\begin{align*}
\norm{S_{a,b}:L_{b}^{p}(u)\to L_{b}^{p}(u)}
=\norm{S_{a,b}:L_{0}^{p}(\rho)\to L_{0}^{p}(\rho)}
=\norm{S_{b,a}:L_{a}^{p'}(\psi)\to L_{a}^{p'}(\psi)},
\end{align*}
and a similar statement holds for $S_{a,b}^{+}$ and 
$S_{b,a}^{+}$. 
Letting $\nu(z)=\psi(z)^{\frac{-p}{p'}}$ we may use the results of 
this section to deduce:
\begin{align}\label{E:dualwt}
[\psi,\nu]_{D_{p',b,a}}^{\frac{1}{2p'}}
&\lesssim \norm{S_{a,b}:L_{b}^{p}(u)\to L_{b}^{p}(u)}
\\&\leq \norm{S_{a,b}^{+}:L_{b}^{p}(u)\to L_{b}^{p}(u)}
\lesssim [\psi,\nu]_{D_{p',b,a}}^{\max\left\{1,\frac{1}{p'-1}\right\}}, 
\end{align}
and this is exactly what is claimed in Theorem \ref{T:1wtsat} for the 
case $b>-1$.

%
%

Recall that 
$\widetilde{u}(z)=u(z)\left(1-\abs{z}^2\right)^{pa}$ and 
$\sigma(z)=u(z)^{\frac{-p'}{p}}$. We will use the following 
fact:
\begin{align*}
\norm{S_{a,b}:L_{b}^{p}(u)\to L_{b}^{p}(u)}
=\norm{Q_{a,b}: L_{b}^{p}(u)\to L_{b}^{p}(\widetilde{u})}
=\norm{Q_{a,b}(\sigma \cdot):L_{b}^{p}(\sigma)\to L_{b}^{p}(\widetilde{u})}.
\end{align*}  
A similar statement of course also holds for $Q_{a,b}^{+}$ and 
$S_{a,b}^{+}$. 

\subsection{Proof of Lower Bound in Theorem \ref{T:1wtsat} when $b+1>0$}
In this subsection we prove the lower bound in Theorem \ref{T:1wtsat} 
under the assumption that $b+1>0$. That is, we will show
\begin{align*}
\mathcal{A}
:=\norm{Q_{a,b}(\sigma \cdot):L_{b}^{p}(\sigma)\to L_{b}^{p}(\widetilde{u})}
<\infty
\hspace{.2in}
\Rightarrow
\hspace{.2in}
[u,\sigma]_{D_{p,a,b}}\lesssim \mathcal{A}^{2p}.
\end{align*}

We first give a familiar property of weights.
\begin{lemma}\label{L:bddbelow}
There holds:
\begin{align*}
\avg{\widetilde{u}}_{\widehat{K_\alpha}}^{d\v{b}}
  \left(\avg{\sigma}_{\widehat{K_\alpha}}^{d\v{b}}\right)^{p-1}
  \vol{b}\left(\widehat{K_\alpha}\right)^{\frac{-pa}{n+1+b}}
\gtrsim 1.
\end{align*}
\end{lemma}
\begin{proof}
Recall that $\widetilde{u}(z)=u(z)(1-\abs{z}^2)^{pa}$ and 
$\widetilde{u}_{b}(\widehat{K_\alpha})=\int_{\widehat{K_\alpha}}
\widetilde{u}(z)d\v{b}(z)$.
To prove the claim, we will prove the equivalent inequality:
\begin{align*}
\vol{b}(\widehat{K_\alpha})
  \vol{b}(\widehat{K_\alpha})^{\frac{a}{n+1+b}}
\lesssim \widetilde{u}_{b}(\widehat{K_\alpha})^{\frac{1}{p}}
  \left(\sigma_{b}(\widehat{K_\alpha})\right)^{\frac{1}{p'}}.
\end{align*}
Indeed, there holds:
\begin{align*}
\vol{b}(\widehat{K_\alpha})
  \vol{b}(\widehat{K_\alpha})^{\frac{a}{n+1+b}}
\simeq (1-\abs{\alpha}^2)^{n+1+b+a}
\simeq \int_{\widehat{K_\alpha}}(1-\abs{z}^{2})^{a}d\v{b}(z).
\end{align*}
So by H\"{o}lder's Inequality we have:
\begin{align*}
\vol{b}(\widehat{K_\alpha})
  \vol{b}(\widehat{K_\alpha})^{\frac{a}{n+1+b}}
\simeq \int_{\widehat{K_\alpha}}\sigma(z)^{\frac{1}{p'}}
  \widetilde{u}(z)^{\frac{1}{p}}d\v{b}(z)
\lesssim \widetilde{u}_{b}(\widehat{K_\alpha})^{\frac{1}{p}}
  \left(\sigma_{b}(\widehat{K_\alpha})\right)^{\frac{1}{p'}}.
\end{align*}
\end{proof}

If $\mathcal{A}<\infty$, then in particular the following weak--type 
inequality holds:
\begin{align*}
\widetilde{u}_{b}\left(\left\{w\in\Bn:\abs{Q_{a,b}(\sigma f)(w)}>
  \lambda\right\}\right)
\lesssim \frac{\mathcal{A}^p}{\lambda^{p}}
  \int_{\Bn}\abs{f(z)}^p\sigma(z) d\v{b}(z).
\end{align*}
Since $n+1+a+b>0$, 
by \cite{Bek1981}*{Lemma 5} there is an $N=N(n,a,b)>0$ 
so that if $\alpha\in\mathcal{T}$ and $d(\alpha)>N$, then there is a 
$\beta\in\mathcal{T}$ with $d(\beta)=d(\alpha)$ such that for 
all $z\in\widehat{K_\beta}$ there holds:
$
\abs{Q_{a,b}(\sigma\unit_{\widehat{K_\alpha}})(z)}
\gtrsim \avg{\sigma}_{\widehat{K_\alpha}}^{d\v{b}}
  \vol{b}\left(\widehat{K_\alpha}\right)^{\frac{-a}{n+1+b}}.
$
Therefore, 
\begin{align*}\widehat{K_\beta}\subset 
\left\{w\in\Bn:\abs{Q_{a,b}(\sigma \unit_{\widehat{K_\alpha}})(w)}\gtrsim
\avg{\sigma}_{\widehat{K_\alpha}}^{d\v{b}}
  \vol{b}\left(\widehat{K_\alpha}\right)^{\frac{-a}{n+1+b}}\right\}.
\end{align*}
By the weak--type inequality, there holds:
\begin{align*}
\widetilde{u}_{b}(\widehat{K_\beta})
\leq \mathcal{A}^{p} 
  \frac{\vol{b}\left(\widehat{K_\alpha}\right)^{\frac{pa}{n+1+b}}
  \vol{b}\left(\widehat{K_\alpha}\right)^{p}}{\sigma_{b}(\widehat{K_{\alpha}})^{p}}
  \int_{\widehat{K_\alpha}}\sigma(z)d\v{b}(z).
\end{align*}
Rearranging this we find:
\begin{align*}
\avg{\widetilde{u}}_{\widehat{K_\beta}}^{d\v{b}}
\left(\avg{\sigma}_{\widehat{K_\alpha}}^{d\v{b}}\right)^{p-1}
\vol{b}\left(\widehat{K_\alpha}\right)^{\frac{-pa}{n+1+b}}
\lesssim \mathcal{A}^{p},
\end{align*}
and interchanging the roles of $\alpha$ and $\beta$ yields:
\begin{align*}
\avg{\widetilde{u}}_{\widehat{K_\alpha}}^{d\v{b}}
\left(\avg{\sigma}_{\widehat{K_\beta}}^{d\v{b}}\right)^{p-1}
\vol{b}\left(\widehat{K_\beta}\right)^{\frac{-pa}{n+1+b}}
\lesssim \mathcal{A}^{p}.
\end{align*}
Thus, using Lemma \ref{L:bddbelow} there holds:
\begin{align}\label{E:smallball}
\sup_{\alpha\in\mathcal{T}:d(\alpha)>N}
\avg{\widetilde{u}}_{\widehat{K_\alpha}}^{d\v{b}}
\left(\avg{\sigma}_{\widehat{K_\alpha}}^{d\v{b}}\right)^{p-1}
\vol{b}\left(\widehat{K_\alpha}\right)^{\frac{-pa}{n+1+b}}
\lesssim \mathcal{A}^{2p}.
\end{align}

This proves the lower bound in Theorem \ref{T:1wtsat} when the supremum 
is taken over small tents. We now show that it holds when the supremum 
is taken over big tents.
Define: $$f(w)=\unit_{B_{\beta}(0,N)}(w)\left(1-\abs{w}^2\right)^{-b}$$ and 
note that $f$ is bounded.
Then since $(1-\overline{z}w)^{-1}$ is analytic as a function of $w$ and has 
no zeros, it follows that $Q_{a,b}f(z) = C_N$.
Using again the weak--type inequality this implies:
\begin{align*}
\widetilde{u}_{b}(\Bn)
=\widetilde{u}_{b}\left(\left\{w\in\Bn:\abs{Q_{a,b}f(w)}>\frac{C_N}{2}\right\}\right)
\leq \frac{2^p\mathcal{A}^{p}}{C_N^p}
  \int_{B_{\beta}(0,N)}\abs{f(z)}^pd\v{b}(z)
\simeq \mathcal{A}^{p}.
\end{align*}
On the other hand, if $Q_{a,b}$ is well--defined for $f\in L_{b}^{p}(u)$, 
then $\sigma(\Bn)<\infty$ by \cite{Bek1981}*{Lemma 4}. Therefore, 
by \eqref{E:smallball} and this observation there holds:
\begin{align*}
\sup_{\alpha\in\mathcal{T}}
\avg{\widetilde{u}}_{\widehat{K_\alpha}}^{d\v{b}}
\left(\avg{\sigma}_{\widehat{K_\alpha}}^{d\v{b}}\right)^{p-1}
\vol{b}\left(\widehat{K_\alpha}\right)^{\frac{-pa}{n+1+b}}
\lesssim \mathcal{A}^{2p},
\end{align*}
as desired. 


\subsection{Proof of Upper Bound in Theorem \ref{T:1wtsat} when $b+1>0$}
In this subsection we prove the upper bound in Theorem \ref{T:1wtsat} 
under the assumption that $b+1>0$. That is, we will show:
\begin{align*}
\norm{Q_{a,b}^{+}(\sigma \cdot):
  L_{b}^{p}(\sigma)\to L_{b}^{p}(\widetilde{u})}
\lesssim [u,\sigma]_{D_{p,a,b}}^{\max\left\{1,\frac{1}{p-1}\right\}}.
\end{align*}

We first handle the case $1<p\leq 2$; the other case will follow from
a duality argument. Fix a Bergman tree $\mathcal{T}$ and let $T=T_{\mathcal{T}}$
where $T_{\mathcal{T}}$ is the operator in \eqref{E:sparse}.
It will be enough to estimate $\norm{T(\sigma \cdot):L_{b}^{p}(\sigma)
\to L_{b}^{p}(\widetilde{u})}$. That is, we will show:
\begin{align*}
\norm{T(\sigma f)}_{L_{b}^{p}(\widetilde{u})}
\lesssim [u,\sigma]_{D_{p,a,b}}^{\frac{1}{p-1}}\norm{f}_{L_{b}^{p}(\sigma)}.
\end{align*}
It is more convenient to prove the equivalent inequality:
\begin{align*}
\norm{T(\sigma f)^{p-1}}_{L_{b}^{p'}(\widetilde{u})}
\lesssim [u,\sigma]_{D_{p,a,b}}\norm{f}_{L_{b}^{p}(\sigma)}^{p-1}.
\end{align*}
We will use duality and prove the following estimate, for 
all non--negative $f\in L_{b}^{p}(\sigma)$ and $g\in L_{b}^{p}(\widetilde{u})$:
\begin{align}\label{E:satDualized}
\ip{T(\sigma f)^{p-1}}{\widetilde{u}g}_{L_{b}^{2}}
\lesssim [u]_{D_{p,a,b}}\norm{f}_{L_{b}^{p}(\sigma)}^{p-1}
  \norm{g}_{L_{b}^{p}(\widetilde{u})}.
\end{align}

Before proving \eqref{E:satDualized}, we discuss some facts that 
will be used. First, using \eqref{E:volstuff} there holds:
\begin{align*}
\vol{b}(\widehat{K_\alpha})^{1+\frac{a}{n+1+b}}
\simeq\vol{b}(K_\alpha)^{1+\frac{a}{n+1+b}}
\simeq \left(1-\abs{\alpha}^2\right)^{n+1+a+b}
\simeq \int_{\widehat{K_\alpha}}\left(1-\abs{z}^2\right)^{a}
  d\v{b}(z).
\end{align*}
Therefore using the fact that $\sigma=u^{\frac{-p'}{p}}$ we have:
\begin{align*}
\vol{b}(\widehat{K_\alpha})^{1+\frac{a}{n+1+b}}
\simeq \int_{{K_\alpha}}
  \sigma(z)^{\frac{1}{p'}}u(z)^{\frac{1}{p}}
  \left(1-\abs{z}^2\right)^{a}d\v{b}(z)
\leq \left(\sigma_{b}({K_\alpha})\right)^{\frac{1}{p'}}
  \left(\widetilde{u}_{b}({K_\alpha})\right)^{\frac{1}{p}}.
\end{align*}
Recall also that $\avg{\sigma f}_{\widehat{K_\alpha}}^{d\v{b}}
=\avg{f}_{\widehat{K_\alpha}}^{\sigma d\v{b}}
\avg{\sigma}_{\widehat{K_\alpha}}^{d\v{b}}$. Finally, since 
$h(x)=x^{r}$ is subadditive for $0<r\leq 1$, using this applied with $r=p-1$, 
recall that $1<p\leq 2$, gives:
\begin{align*}
\left(T(\sigma f)\right(z))^{p-1}
\leq \sum_{\alpha\in\mathcal{T}}
  \vol{b}\left(\widehat{K_\alpha}\right)^{\frac{-a(p-1)}{n+1+b}}
  \left(\avg{\sigma f}_{\widehat{K_\alpha}}^{d\v{b}}\right)^{p-1}
  \unit_{\widehat{K_\alpha}}(z).
\end{align*}
Using these facts, we now prove \eqref{E:satDualized}.   Indeed, we have
\begin{align}
\ip{T(\sigma f)^{p-1}}{\widetilde{u}g}_{L_{b}^{2}}
&\leq\sum_{\alpha\in\mathcal{T}}
  \left(\avg{f}_{\widehat{K_\alpha}}^{\sigma d\v{b}}
  \avg{\sigma}_{\widehat{K_\alpha}}^{d\v{b}}
  \vol{b}(\widehat{K_\alpha})^{\frac{-a}{n+1+b}}\right)^{p-1}
  \int_{\widehat{K_\alpha}}g(z)\widetilde{u}(z)d\v{b}(z)
\\&\lesssim\sum_{\alpha\in\mathcal{T}}
  \left(\avg{f}_{\widehat{K_\alpha}}^{\sigma d\v{b}}\right)^{p-1}
  \avg{g}_{\widehat{K_\alpha}}^{\widetilde{u}d\v{b}}
  \frac{
  \left(\avg{\sigma}_{\widehat{K_\alpha}}^{d\v{b}}\right)^{p-1}
  \avg{\widetilde{u}}_{\widehat{K_\alpha}}^{d\v{b}}}
  {\vol{b}\left(\widehat{K_\alpha}\right)^{\frac{pa}{n+1+b}}}
  \vol{b}(\widehat{K_\alpha})^{1+\frac{a}{n+1+b}}
\\&\leq [u,\sigma]_{D_{p,a,b}}
  \sum_{\alpha\in\mathcal{T}}
  \left(\avg{f}_{\widehat{K_\alpha}}^{\sigma d\v{b}}\right)^{p-1}
  \left(\sigma_{b}({K_\alpha})\right)^{\frac{1}{p'}}
  \avg{g}_{\widehat{K_\alpha}}^{\widetilde{u}d\v{b}}
  \left(\widetilde{u}_{b}({K_\alpha})\right)^{\frac{1}{p}}
  \label{E:sabest}.
\end{align}
By H\"{o}lder's Inequality, the sum above is dominated by:
\begin{align*}
\left\{\sum_{\alpha\in\mathcal{T}}
  \left(\avg{f}_{\widehat{K_\alpha}}^{\sigma d\v{b}}\right)^{p}
  \sigma_{b}({K_\alpha})\right\}^{\frac{p-1}{p}}
\left\{\sum_{\alpha\in\mathcal{T}}
  \left(\avg{g}_{\widehat{K_\alpha}}^{\widetilde{u} d\v{b}}\right)^{p}
  \widetilde{u}_{b}({K_\alpha})\right\}^{\frac{1}{p}}.
\end{align*}
Using the disjointness of the sets $K_\alpha$ we estimate the first 
factor above using Lemma \ref{L:maxbdd}:
\begin{align*}
\sum_{\alpha\in\mathcal{T}}
  \left(\avg{f}_{\widehat{K_\alpha}}^{\sigma d\v{b}}\right)^{p}
  \sigma_{b}({K_\alpha})
\leq\int_{\Bn}\left(M_{\mathcal{T},\sigma d\v{b}}f(z)\right)^{p}
  \sigma(z) d\v{b}(z)
\leq \norm{f}_{L_{b}^{p}(\sigma)}^{p}.
\end{align*}
A similar estimate holds for the second factor, completing the proof in the case 
$1<p\leq 2$.

We now handle the case $2<p<\infty$. That is we want to show:  
\begin{align}\label{E:berDualizedDualc2}
\ip{T(\sigma f)}{\widetilde{u}g}_{L_{b}^{2}}
\lesssim [u,\sigma]_{D_{p,a,b}}\norm{f}_{L_{b}^{p}(\sigma)}
  \norm{g}_{L_{b}^{p'}(\widetilde{u})},
\end{align}
for all non--negative $f\in L_{b}^p(\sigma)$ and 
$g\in L_{b}^{p'}(\widetilde{u})$. Now, define:
\begin{align*}
\psi(z)
:={\sigma(z)}{\left(1-\abs{z}^2\right)^{-p'a}}
\hspace{.2in}
\textnormal{ and }
\hspace{.2in}
\widetilde{\psi}(z)
:=\psi(z)(1-\abs{z}^2)^{p'a}.
\end{align*}
Clearly, $\widetilde{\psi}=\sigma$. Set $\rho(z):=\psi(z)^{\frac{-p}{p'}}$.
There holds:
\begin{align*}
\widetilde{u}(z)
=\sigma(z)^{\frac{-p}{p'}}(1-\abs{z}^2)^{pa}
=\left(\sigma(z)(1-\abs{z}^2)^{-p'a}\right)^{\frac{-p}{p'}}
=\rho(z).
\end{align*}
Now, it is easy to see that:
\begin{align*}
[\psi,\rho]_{D_{p',a,b}}
&:=\sup_{\alpha\in\mathcal{T}}
  \left(\frac{\int_{\widehat{K_\alpha}}\rho d\v{b}}
  {\vol{b}(\widehat{K_\alpha})}\right)^{p'-1}
  \left(\frac{\int_{\widehat{K_\alpha}}\widetilde{\psi}d\v{b}}
  {\vol{b}(\widehat{K_\alpha})}\right)
  \vol{b}(\widehat{K_\alpha})^{\frac{-p'a}{n+1+b}}
\\&=\sup_{\alpha\in\mathcal{T}}
  \left(\frac{\int_{\widehat{K_\alpha}}\widetilde{u} d\v{b}}
  {\vol{b}(\widehat{K_\alpha})}\right)^{p'-1}
  \left(\frac{\int_{\widehat{K_\alpha}}\sigma d\v{b}}
  {\vol{b}(\widehat{K_\alpha})}\right)
  \vol{b}(\widehat{K_\alpha})^{\frac{-p'a}{n+1+b}}
\\&=[u,\sigma]_{D_{p,a,b}}^{p'-1}.
\end{align*}
Therefore, using the fact that $T$ is self--adjoint 
we have $\ip{T(\sigma f)}{\widetilde{u}g}_{L_{b}^{2}}
=\ip{T(\rho g)}{\widetilde{\psi}f}_{L_{b}^{2}}$. Using the 
fact that $p'<2$, yields: 
\begin{align*}
\ip{T(\rho g)}{\widetilde{\psi}f}_{L_{b}^{2}}
\leq [\psi,\rho]_{D_{p',a,b}}^{\frac{1}{p'-1}}
  \norm{g}_{L_{b}^{p'}(\rho)}
  \norm{f}_{L_{b}^{p}(\widetilde{\psi})}
=[u,\sigma]_{D_{p,a,b}}\norm{f}_{L_{b}^{p}(\sigma)}
  \norm{g}_{L_{b}^{p'}(\widetilde{u})}.
\end{align*}
This completes the proof of the upper bound in Theorem \ref{T:1wtsat}.

%

\section{A Sharp Example}\label{sharp}
In this section, we give a weight, $u$ and a function $f$ such that:
\begin{align*}
\norm{Pf}_{L_{b}^2(u)}
\gtrsim [u]_{B_2}\norm{f}_{L_{b}^2(u)},
\end{align*}
which implies that the upper bound in Theorem \ref{T:1wtsat} is 
sharp. The idea is to reduce to the one--dimensional case and to 
use what is essentially the sharp example in \cite{PottReg2013}.

Let $u(z)=u(z_1)=\abs{1-z_{1}}^{(n+1+b)(1-\delta)}\abs{1+z_{1}}^{(n+1+b)(\delta-1)}$. 
We want to compute the $B_2$ characteristic of $u$. For $r_0 > 0$, abuse 
notation and let $r_0$ denote the vector $(r_0,0,\ldots,0)$ and similarly 
for $-r_0$. Let $z=(z_1,z')$, that is $z'=(z_2,\ldots,z_n)$. There holds:
\begin{align}\label{E:sharInt}
\int_{T_{r_0}}u(z)(1-\abs{z}^2)^bdV(z)
\end{align} 
is equal to
\begin{align}\label{E:sharInt1}
\int_{\{z_1:\abs{1-z_1}<1-r_0\}}
  \frac{\abs{1-z_1}^{(n+1+b)(1-\delta))}}{\abs{1+z_1}^{(n+1+b)(1-\delta))}}
  \int_{z':\abs{z'}^2<1-\abs{z_1}^2}
  (1-\abs{z}^2)^bdV_{n-1}(z')dA(z_1),
\end{align}
where above $dV_{n-1}$ is Lebesgue measure on $\mathbb{C}^{n-1}$ and $dA(z)$ is
Lebesgue measure on $\mathbb{C}$. For the inner integral, let 
$w=z'/\sqrt{1-\abs{z_1}^2}$. Using this change of variables, the inner integral becomes:
\begin{align*}
\int_{\B_{n-1}} (1-\abs{z_1}^2)^b(1-\abs{w}^2)^b
  \left(\sqrt{1-\abs{z_1}^2}\right)^{2(n-1)}
  dV_{n-1}(w)
\simeq \left(1-\abs{z_1}^2\right)^{n+b-1}.
\end{align*}
Inserting this into \eqref{E:sharInt1}, we see that \eqref{E:sharInt} is
comparable to:
\begin{align*}
\int_{\{z_1\in\D:\abs{1-z_1}<1-r_0\}}u(z_1)
  \left(1-\abs{z_1}^2\right)^{n+b-1}dA(z_1).
\end{align*}
To estimate this integral, it is easiest to make the conformal change of
variables $w=i\frac{1-z_1}{1+z_1}$ so that when $r_0$ is bounded away 
from $0$ (as is the case here) this integral is comparable to:
\begin{align*}
\int_{\{w\in \mathbb{H}: \abs{w}<R(r_0)\}}
  \abs{w}^{(n+1+b)(1-\delta)}(\Im w)^{n+b-1}dA(w)
\simeq \frac{R_0^{(n+1+b)(2-\delta)}}{(n+1+b)(2-\delta)},
\end{align*}
where above $R_0=R(r_0)$. 

Using similar reasoning, there holds:
\begin{align}
\int_{T_{r_0}}u^{-1}(z)(1-\abs{z}^2)^bdV(z)
&\simeq \int_{\{w\in \mathbb{H}: \abs{w}<R(r_0)\}}
  \abs{w}^{(n+1+b)(\delta-1)}(\Im w)^{n+b-1}dA(w)
\\&\simeq \frac{R_0^{(n+1+b)\delta}}{(n+1+b)\delta}.
\end{align} 
With $R=R(r_0)$ there holds $\v{b}(T_{r_0})\simeq R^{n+1+b}$. 
Thus, there holds $\avg{u}_{T_{r_0}}^{d\v{b}}\avg{u^{-1}}_{T_{r_0}}^{d\v{b}}
\simeq \delta^{-1}$. Similarly, $\avg{u}_{T_{r_0}}^{d\v{b}}
\avg{u^{-1}}_{T_{-r_0}}^{d\v{b}}
\simeq \delta^{-1}$. Now, the singularities of $u$ and $u^{-1}$ are at 
$(-1,0,\ldots,0)$ and $(1,0,\ldots,0)$ so the argument above implies that
if we take, say, $r_0>\frac{1}{2}$ (so that $R_0\lesssim 1$), then $u$ is a
$B_2$ weight with $[u]_{B_2}\simeq \delta^{-1}$. 

Now, let $f(w)=u^{-1}(w)\unit_{T_{1/2}}$. Then $\int_{\Bn}\abs{f(w)}^2ud\v{b}(w) = 
\int_{T_{1/2}}u^{-1}(w)d\v{b}(w)\simeq\delta^{-1}$. Now, we give a pointwise estimate 
of $P_{b}f(z)$. To do this, we may use the idea that we used to obtain the lower
bound in Theorem \ref{T:1wtsat} from \cite{Bek1981}. That is, for 
$z\in T_{-1/2}$ there holds $\abs{P_{b}f(z)} \geq \avg{f}_{T_{1/2}}
\simeq \delta^{-1}$.
Therefore, making the change of variables $w'=-w$ there holds:
\begin{align*}
\norm{Pf}_{L_{b}^2(u)}^2
&=\int_{\Bn}\abs{Pf(w)}^2ud\v{b}(w)
\\&\geq \delta^{-2}\int_{T_{-1/2}}u(w)d\v{b}(w)
\\&= \delta^{-2}\int_{T_{1/2}}u^{-1}(w')d\v{b}(w')
= [u]_{B_2}^{2}\norm{f}_{L_{b}^2(u)}^2.
\end{align*}

\section{Conclusion}\label{concl}
The subject of this paper has been one weight inequalities for operators 
acting on function spaces defined on $\Bn$. There are at least two 
additional directions in which one may continue this line of research. The 
first is proving results like the ones in this paper for more general domains. 
The second is proving two--weight inequalities. That is, if $T$ is one 
of the operators discussed in this paper, for which weights
$w,\sigma$ do we have $\norm{T:L_{b}^{p}(w)\to L_{b}^{p}(\sigma)}$ is 
finite?

\end{document}